\documentclass[a4paper,10pt]{amsart}      
\usepackage{amssymb,amsthm, amsmath, amsfonts} 
\usepackage{graphics}                 
\usepackage{color}                    
\usepackage{hyperref}                 
\usepackage[all]{xy}
\usepackage{enumerate}

\newtheorem{theorem}{Theorem}[section]
\newtheorem{lemma}[theorem]{Lemma}
\newtheorem{proposition}[theorem]{Proposition}

\newtheorem{corollary}[theorem]{Corollary}

\newtheorem{example}[theorem]{Example}

\numberwithin{equation}{section}
\def\rad{\mathop{\hbox{rad}}}
\def\Ob{\mathop{\hbox{Ob}}}
\def\pd{\mathop{\hbox{pd}}}
\def\dim{\mathop{\hbox{dim}}}
\def\top{\mathop{\hbox{Top}}}

\title{Algebras stratified for all linear orders}
\author{Liping Li}
\address{School of Mathematics, University of Minnesota, MN, 55455, USA}
\email{lixxx480@math.umn.edu}
\thanks{The author would like to thank his thesis advisor, Professor Peter Webb, for many invaluable suggestions in numerous discussions. He also thanks Professor Mazorchuk for pointing out some works on relevant topics, which were unknown to the author before. The referee is appreciated for the careful reading and many nice suggestions.}
\begin{document}

\begin{abstract}
In this paper we describe several characterizations of basic finite-dimensional $\Bbbk$-algebras $A$ stratified for all linear orders, and classify their graded algebras as tensor algebras satisfying some extra property. We also discuss whether for a given linear order $\preccurlyeq$, $\mathcal{F} (_{\preccurlyeq} \Delta)$, the category of $A$-modules with $_{\preccurlyeq} \Delta$-filtrations, is closed under cokernels of monomorphisms, and classify quasi-hereditary algebras satisfying this property.
\end{abstract}
\maketitle

Let $A$ be a basic finite-dimensional $\Bbbk$-algebra where $\Bbbk$ is an algebraically closed filed, and $A$-mod be the category of finitely generated left $A$-modules. Dlab and Ringel showed in \cite{Dlab2} that $A$ is quasi-hereditary for all linear orders if and only if $A$ is a hereditary algebra. In \cite{Frisk1} stratification property of $A$ for different orders was studied. These results motivate us to classify algebras standardly stratified or properly stratified for all linear orders, which include hereditary algebras as special cases.

Choose a fixed set of orthogonal primitive idempotents $\{ e_{\lambda} \}_{\lambda \in \Lambda}$ for $A$ such that $\sum _{\lambda \in \Lambda} e_{\lambda}= 1$. We can define a $\Bbbk$-linear category $\mathcal{A}$ as follows: $\Ob \mathcal{A} = \{ e_{\lambda} \} _{\lambda \in \Lambda}$; for $\lambda, \mu \in \Lambda$, the morphism space $\mathcal{A} (e_{\lambda}, e_{\mu}) = e_{\mu} A e_{\lambda}$. The category $\mathcal{A}$ is skeletal, and the endomorphism algebra of each object is local. A $\ \Bbbk$-linear representation of $\mathcal{A}$ is a $\Bbbk$-linear functor from $\mathcal{A}$ to the category of finite-dimensional $\Bbbk$-vector spaces. It is clear that $\mathcal{A}$-rep, the category of finite-dimensional $\Bbbk$-linear representations of $\mathcal{A}$, is equivalent to $A$-mod. Call $\mathcal{A}$ the \textit{associated category} of $A$.

we prove that if $A$ is standardly stratified for all linear orders, then its associated category is a \textit{directed category}, i.e., there is a partial order $\leqslant$ on $\Lambda$ such that $\lambda \leqslant \mu$ whenever $\mathcal{A} (e_{\lambda}, e_{\mu}) = e_{\mu} A e_{\lambda} \neq 0$. Please see \cite{Li1, Li2} for details of directed categories. With this terminology, we get:

\begin{theorem}
Let $A$ be a basic finite-dimensional algebra and $\mathcal{A}$ be its associated category. Then the following are equivalent:
\begin{enumerate}
\item $A$ is standardly stratified for all linear orders.
\item $\mathcal{A}$ is a directed category and $J = \bigoplus _{\lambda \neq \mu \in \Lambda} e_{\mu} A e_{\lambda}$ is a projective $A$-module.
\item The trace tr$ _{P_{\lambda}} (P_{\mu})$ is a projective module for $\lambda, \mu \in \Lambda$.
\item The projective dimension $\pd _A M \leqslant 1$ for all $M \in \mathcal{F} (_{\preccurlyeq} \Delta)$ and linear orders $\preccurlyeq$.
\end{enumerate}
\end{theorem}

Note that $J = \bigoplus _{\lambda \neq \mu \in \Lambda} e_{\mu} A e_{\lambda}$ can be viewed as a two-sided ideal of $A$ with respect to this chosen set of orthogonal primitive idempotents since $\mathcal{A}$ is a directed category. Moreover, we have $A = A_0 \oplus J$ as vector spaces, where $A_0 = \bigoplus _{x \in \text{Ob } \mathcal{A}} \mathcal{A} (x, x)$ constitutes of all endomorphisms in $\mathcal{A}$. Thus we consider the associated graded algebra $\check{A}$ with $\check{A}_0 = A_0$ and $\check{A}_i = J^i /J^{i+1}$ for $i \geqslant 1$, and obtain the following classification:

\begin{theorem}
Let $A$ be a basic finite-dimensional $\Bbbk$-algebra whose associated category $\mathcal{A}$ is directed. Then the following are equivalent:
\begin{enumerate}
\item $A$ is standardly stratified (resp., properly stratified) for all linear orders;
\item the associated graded algebra $\check{A}$ is standardly stratified (resp., properly stratified) for all linear orders;
\item $\check{A}$ is the tensor algebra generated by $A_0 = \bigoplus _{\lambda \in \Lambda} \mathcal{A} (e_{\lambda}, e_{\lambda}) = e_{\lambda} A e_{\lambda}$ and a left (resp., left and right) projective $A_0$-module $\check{A}_1$.
\end{enumerate}
\end{theorem}

We also describe a combinatorial construction for $\check {\mathcal{A}}$, the associated graded category of $\check{A}$.

Let $\preccurlyeq$ be a particular linear order for which $\mathcal{A}$ is standardly stratified. It is well known that $\mathcal{F} (_{\preccurlyeq} \Delta)$ is an additive category closed under extensions, direct summands, and kernels of epimorphisms (see \cite{Cline, Dlab1, Webb}). But in general it is not closed under cokernels of monomorphisms. However, if $A$ is standardly stratified with respect to all linear orders, there exists a particular linear order (not necessarily unique) $\preccurlyeq$ for which the corresponding category $\mathcal{F} (_{\preccurlyeq} \Delta)$ is closed under cokernels of monomorphisms. It motivates us to give a criterion for whether $\mathcal{F} (_{\preccurlyeq} \Delta)$ has this property. We obtain the following result:

\begin{theorem}
Let $A$ be a finite-dimensional basic algebra standardly stratified for a linear order $\preccurlyeq$. Then:
\begin{enumerate}
\item $\mathcal{F} (_{\preccurlyeq} \Delta)$ is closed under cokernels of monomorphisms if and only if the cokernel of every monomorphism $\iota: \Delta_{\lambda} \rightarrow P$ is contained in $\mathcal{F} (_{\preccurlyeq} \Delta)$, where $P$ is an arbitrary projective module and $\lambda \in \Lambda$.
\item If $\mathcal{F} (_{\preccurlyeq} \Delta)$ is closed under cokernels of monomorphisms and $A$ is standardly stratified for another linear order $\preccurlyeq'$, then $\mathcal{F} (_{\preccurlyeq'} \Delta) \subseteq \mathcal{F} (_{\preccurlyeq} \Delta)$.
\item If $A$ is quasi-hereditary, then $\mathcal{F} (_{\preccurlyeq} \Delta)$ is closed under cokernels of monomorphisms if and only if $A$ is a quotient of a finite-dimensional hereditary algebra and all standard modules are simple.
\end{enumerate}
\end{theorem}

Thus if $A$ is standardly stratified for $\preccurlyeq$ and $\preccurlyeq'$ such that both $\mathcal{F} (_{\preccurlyeq} \Delta)$ and $\mathcal{F} (_{\preccurlyeq'} \Delta)$ are closed under cokernels of monomorphisms, then $\mathcal{F} (_{\preccurlyeq} \Delta) = \mathcal{F} (_{\preccurlyeq'} \Delta)$.

In practice it is hard to determine whether there exists a linear order $\preccurlyeq$ for which $A$ is standardly stratified and the corresponding category $\mathcal{F} (_{\preccurlyeq} \Delta)$ is closed under cokernels of monomorphisms. Certainly, checking all linear orders is not an ideal way to do this. We then describe an explicit algorithm to construct a set $\mathcal{L}$ of linear orders with respect to all of which $A$ is standardly stratified. Moreover, if there exists a linear order $\preccurlyeq$ such that $A$ is standardly stratified and the corresponding category $\mathcal{F} (_{\preccurlyeq} \Delta)$ is closed under cokernels of monomorphisms, then $\preccurlyeq \in \mathcal{L}$.

This paper is organized as follows. In Section 1 we describe some equivalent conditions for $A$ to be standardly stratified with respect to all linear orders, and prove the first theorem. The classification is described in Section 2 with a proof of the second theorem. The problem that whether $\mathcal{F} (_{\preccurlyeq} \Delta)$ is closed under cokernels of monomorphisms is considered in Section 3. In the last section we describe the algorithm and give some examples.

We introduce some conventions here. All algebras in this paper are finite-dimensional basic $\Bbbk$-algebras, and all modules are finitely generated left modules if we do not make other assumptions. Composition of morphisms is from right to left. The field $\Bbbk$ is supposed to be algebraically closed. We view the zero module 0 as a projective (or free) module since this will simplify the expressions and proofs of many statements.

\section{Characterizations of algebras stratified for all linear orders}

We start with some preliminary knowledge on stratification theory, for which the reader can find more details in \cite{Cline, Dlab1, Dlab2, Webb, Xi}. Throughout this section $A$ is a basic finite-dimensional $\Bbbk$-algebra with a chosen set of orthogonal primitive idempotents $\{ e_{\lambda} \} _{\lambda \in \Lambda}$ indexed by a set $\Lambda$ such that $\sum _{\lambda \in \Lambda} e_{\lambda} = 1$. Let $P_{\lambda} = A e_{\lambda}$ and $S_{\lambda} = P_{\lambda} / \rad P_{\lambda}$. According to \cite{Cline}, $A$ is \textit{standardly stratified} with respect to a fixed preorder $\preccurlyeq$ if there exist modules $\Delta_{\lambda}$, $\lambda \in \Lambda$, such that the following conditions hold:
\begin{enumerate}
\item the composition factor multiplicity $[\Delta_{\lambda} : S_{\mu}] = 0$ whenever $\mu \npreceq \lambda$;
and
\item for every $\lambda \in \Lambda$ there is a short exact sequence $0 \rightarrow K_{\lambda} \rightarrow P_{\lambda} \rightarrow \Delta_{\lambda} \rightarrow 0$ such that $K_{\lambda}$ has a filtration with factors $\Delta_{\mu}$ where $\mu \succ \lambda$.
\end{enumerate}
The \textit{standard module} $\Delta_{\lambda}$ is the largest quotient of $P_{\lambda}$ with only composition factors $S_{\mu}$ such that $\mu \preccurlyeq \lambda$. Similarly, the \textit{proper standard module} $\bar{ \Delta}_{\lambda}$ can be defined as the largest quotient of $P_{\lambda}$ such that all composition factors $S_{\mu}$ satisfy $\mu \succ \lambda$ except for one copy of $S_{\lambda}$. \textit{Costandard module} $\nabla _{\lambda}$ and \textit{proper costandard module} $\bar {\nabla} _{\lambda}$ are defined dually. Since by Proposition 1.1 $A$ is standardly stratified for all preorders if and only if it is a direct sum of local algebras, in this paper we only deal with linear orders instead of general preorders as in \cite{Frisk2, Frisk3, Webb} to get a class of algebras whose size is big enough to contain hereditary algebras.

In the case that $A$ is standardly stratified, standard modules $\Delta_{\lambda}$ have the following description:
\begin{equation*}
\Delta_{\lambda} = P_{\lambda} / \sum _{\mu \succ \lambda} \text{tr} _{P_{\mu}} (P_{\lambda}),
\end{equation*}
where tr$_{P_{\mu}} (P_{\lambda})$ is the trace of $P_{\mu}$ in $P _{\lambda}$ (\cite{Dlab1, Webb}). Let $\Delta$ be the direct sum of all standard modules and $\mathcal{F} (\Delta)$ be the full subcategory of $A$-mod such that each object in $\mathcal{F} (\Delta)$ has a filtration by standard modules. Similarly we define $\mathcal{F} (\overline {\Delta})$, $\mathcal{F} (\nabla)$ and $\mathcal{F} (\overline {\nabla})$. According to Theorem 3.4 in \cite{Webb}, $A$ is standardly stratified if and only if $_AA \in \mathcal{F} (\Delta)$. The algebra is called \textit{properly stratified} if $_AA \in \mathcal{F} (\Delta) \cap \mathcal{F} (\overline {\Delta})$. If the endomorphism algebra of each standard module is one-dimensional, then $A$ is called \textit{quasi-hereditary}.

\begin{proposition}
The algebra $A$ is standardly stratified for all preorders if and only if $A$ is a direct sum of local algebras.
\end{proposition}

\begin{proof}
If $A$ is a direct sum of local algebras, it is standardly stratified for all preorders such that all standard modules coincide with indecomposable projective modules. Conversely, suppose that $A$ is standardly stratified for all preorders. To show that $A$ is a direct sum of local algebras, it is sufficient to show that Hom$ _A (P_{\lambda}, P_{\mu}) \cong e_{\lambda} A e_{\mu} = 0$ for all pairs of distinct elements $\lambda \neq \mu \in \Lambda$.

Consider the preorder on $\Lambda$ such that every two different elements cannot be compared. By the second condition in the definition of standardly stratified algebras, $\Delta_{\lambda} \cong P_{\lambda}$ for all $\lambda \in \Lambda$. By the first condition, $P_{\lambda}$ only has composition factors $S_{\lambda}$. Therefore, Hom$ _A (P_{\lambda}, P_{\mu}) \cong e_{\lambda} A e_{\mu} = 0$ for all $\lambda \neq \mu \in \Lambda$.
\end{proof}

The following statement is an immediate corollary of Dlab's theorem (\cite{Frisk3, Webb}).

\begin{proposition}
The algebra $A$ is properly stratified for a linear order $\preccurlyeq$ on $\Lambda$ if and only if both $A$ and $A^{op}$ are standardly stratified with respect to $\preccurlyeq$, in other words, $A$ is both left and right standardly stratified.
\end{proposition}

\begin{proof}
The algebra $A$ is properly stratified if and only if $\mathcal{F} (\Delta) \cap \mathcal{F} (\overline {\Delta})$ contains all left projective $A$-modules. By duality, this is true if and only if $\mathcal{F} (\Delta)$ contains all left projective $A$-modules and $\mathcal{F} (\overline {\nabla} _A)$ contains all right injective $A$-modules, where $\overline {\nabla} _A$ is a right $A$-module. By Theorem 3.4 in \cite{Webb}, we conclude that $A$ is properly stratified if and only if $\mathcal{F} (\Delta)$ contains all left projective $A$-module and $\mathcal{F} (\Delta_A)$ contains all right projective $A$-modules, here $\Delta_A$ is the direct sum of all right standard modules. That is, $A$ is properly stratified if and only if it is both left and right standardly stratified.
\end{proof}

Recall the associated category $\mathcal{A}$ of $A$ is a \textit{directed} category if there is a partial order $\leqslant$ on $\Lambda$ such that $\mathcal{A} (e_{\lambda}, e_{\mu}) = e_{\mu} A e_{\lambda} \neq 0$ implies $\lambda \leqslant \mu$.

\begin{proposition}
If $A$ is standardly stratified for all linear orders on $\Lambda$, then the associated category $\mathcal{A}$ is a directed category.
\end{proposition}

\begin{proof}
Suppose that the conclusion is not true. Then there is an oriented cycle $e_0 \rightarrow e_1 \rightarrow \ldots \rightarrow e_n = e_0$ with $n > 1$ such that $\mathcal{A} (e_i, e_{i+1}) = e_{i+1} A e_i \neq 0$ for all $0 \leqslant i \leqslant n-1$. Take some $e_s$ such that $\dim _{\Bbbk} P_s \geqslant \dim _{\Bbbk} P_i$ for all $0 \leqslant i \leqslant n-1$, where $P_s = A e_s$ coincides with the vector space formed by all morphisms starting from $e_s$. Then for an arbitrary linear order $\preccurlyeq$ with respect to which $e_s$ is maximal, we claim that $A$ is not standardly stratified.

Indeed, if $A$ is standardly stratified with respect to $\preccurlyeq$, then tr$ _{P_s} (P_i) \cong P_s^{m_i}$ for some $m_i \geqslant 0$. Consider $P_{s-1}$ (if $s=0$ we consider $P_{n-1}$). Since $\mathcal{A} (e_{s-1}, e_s) = e_s A e_{s-1} \neq 0$, tr$ _{P_s} (P_{s-1}) \neq 0$, so $m_{s-1} \geqslant 1$. But on the other hand, since tr$ _{P_s} (P_{s-1}) \subseteq \rad P_{s-1}$, and $\dim _{\Bbbk} P_s \geqslant \dim _{\Bbbk} P_{s-1}$, we should have $m_{s-1} = 0$. This is absurd. Therefore, there is no oriented cycle in $\mathcal{A}$, and the conclusion follows.
\end{proof}

The following proposition describes some properties of algebras with directed associated categories.

\begin{proposition}
Suppose that the associated category $\mathcal{A}$ of $A$ is a directed category with respect to a partial order $\leqslant$. Then
\begin{enumerate}
\item The standard modules $_{\leqslant} \Delta_{\lambda} \cong \mathcal{A} (e_{\lambda}, e_{\lambda}) = e_{\lambda} A e_{\lambda}$ for all $\lambda \in \Lambda$.
\item $A$ is standardly stratified for $\leqslant$ if and only if for every pair $\lambda, \mu \in \Lambda$, $e_{\mu} A e_{\lambda}$ is a free $e_{\mu} A e_{\mu}$-module.
\item The category $\mathcal{F} (_{\leqslant} \Delta)$ is closed under cokernels of monomorphisms.
\end{enumerate}
\end{proposition}

\begin{proof}
The first statement is Proposition 5.5 in \cite{Li1}. The second statement comes from Theorem 5.7 in that paper. We reminder the reader that $\mathcal{A} (e_{\mu}, e_{\mu})$ is a local algebra, so projective modules coincide with free modules. It remains to show the last statement. That is, for an arbitrary exact sequence with $L, M \in \mathcal{F} (_{\leqslant}  \Delta)$, we need to show $N \in \mathcal{F} (_{\leqslant} \Delta)$ as well:
\begin{equation*}
\xymatrix{ 0 \ar[r] & L \ar[r] & M \ar[r] & N \ar[r] & 0.}
\end{equation*}

By the structure of standard modules, it is clear that an $A$-module $K \in \mathcal{F} (_{\leqslant} \Delta)$ if and only if $e_{\lambda} K$ is a free $\mathcal{A} (e_{\lambda},e_{\lambda}) = e_{\lambda} A e_{\lambda}$-module for all $\lambda \in \Lambda$. For an arbitrary $\lambda \in \Lambda$, the above sequence induces an exact sequence of $e_{\lambda} A e_{\lambda}$-modules
\begin{equation*}
\xymatrix{ 0 \ar[r] & e_{\lambda} L \ar[r] & e_{\lambda} M \ar[r] & e_{\lambda} N \ar[r] & 0.}
\end{equation*}
Since $L, M \in \mathcal{F} (_{\leqslant} \Delta)$, we know that $e_{\lambda} L \cong (e_{\lambda} A e_{\lambda})^l$ and $e_{\lambda} M \cong (e_{\lambda} A e_{\lambda})^m$ for some $l, m \geqslant 0$.

Notice that $e_{\lambda} A e_{\lambda}$ is a local algebra. The regular module over $e_{\lambda} A e_{\lambda}$ is indecomposable, has a simple top and finite length. Therefore, the top of $e_{\lambda} L$ is embedded into the top of $e_{\lambda} M$ since otherwise the first map cannot be injective. Therefore, the top of $e_{\lambda} N$ is isomorphic to $S_{\lambda}^{m-l}$ where $S_{\lambda} \cong (e_{\lambda} A e_{\lambda}) / \rad (e_{\lambda} A e_{\lambda})$ and $e_{\lambda} N$ is the quotient module of $(e_{\lambda} A e_{\lambda}) ^{m-l}$. But by comparing dimensions we conclude that $e_{\lambda} N \cong (e_{\lambda} A e_{\lambda}) ^{m-l}$. Consequently, $e_{\lambda} N$ and hence $N$ are contained in $\mathcal{F} (_{\leqslant} \Delta)$ as well. This finishes the proof.
\end{proof}

The second statement of this proposition is the main condition of \textit{pyramid algebras} mentioned in \cite{KM}. The third statement motivates us to study the problem of whether $\mathcal{F} (_{\preccurlyeq} \Delta)$ is closed under cokernels of monomorphisms for an arbitrary linear order $\preccurlyeq$. We will return to this topic in Section 3.

In some sense the property of being standardly stratified for all linear orders is inherited by quotient algebras. Explicitly,

\begin{lemma}
If $A$ is standardly stratified for all linear orders, then for an arbitrary primitive idempotent $e_{\lambda}$, the quotient algebra $\bar{A} = A / A e_{\lambda} A$ is standardly stratified for all linear orders.
\end{lemma}

\begin{proof}
Let $\preccurlyeq$ be a linear order on $\Lambda \setminus \{\lambda \}$. Then we can extend it to a linear order $\tilde {\preccurlyeq}$ on $\Lambda$ by letting $\lambda$ be the unique maximal element. Since $A$ is standardly stratified for $\tilde {\preccurlyeq}$, we conclude that $\bar {A}$ is standardly stratified for $\preccurlyeq$ as well.
\end{proof}

Now suppose that the associated category $\mathcal{A}$ of $A$ is directed. Then we define
\begin{equation*}
J = \bigoplus _{ \lambda \neq \mu \in \Lambda} \mathcal{A} (e_{\lambda}, e_{\mu} ) = \bigoplus _{ \lambda \neq \mu \in \Lambda} e_{\mu} A e_{\lambda}, \qquad A_0 = \bigoplus _{ \lambda \in \Lambda} \mathcal{A} (e_{\lambda}, e_{\lambda}) = \bigoplus _{ \lambda \in \Lambda} e_{\lambda} A e_{\lambda}.
\end{equation*}
Clearly, $J$ is a two-sided ideal of $A$ with respect to this chosen set of orthogonal primitive idempotents, and $A_0 \cong A /J$ as $A$-modules.

\begin{proposition}
If $A$ is standardly stratified for all linear orders on $\Lambda$, then $J$ is a projective $A$-module.
\end{proposition}

\begin{proof}
Notice that the associated category $\mathcal{A}$ is directed by Proposition 1.3, so $J$ is a well defined $A$-module. Let $\leqslant$ be a partial order on $\Lambda$ with respect to which $\mathcal{A}$ is directed, i.e., $\mathcal{A} (e_{\lambda}, e_{\mu}) = e_{\mu} A e_{\lambda} \neq 0$ implies $\lambda \leqslant \mu$. We can extend this partial order to a linear order with respect to which $\mathcal{A}$ is still directed. Indeed, let $O_1$ be the maximal objects in $\mathcal{A}$ with respect to $\leqslant$, $O_2$ be the maximal objects in $\Ob \mathcal{A} \setminus O_1$ with respect to $\leqslant$, and so on. In this way we get $\Ob \mathcal{A} = \bigsqcup _{i=1}^n O_i$. Then define arbitrary linear orders $\leqslant_i$ on $O_i$, $1 \leqslant i \leqslant n$. Take $x, y \in \Ob \mathcal{A}$ and suppose that $x \in O_i$ and $y \in O_j$, $1 \leqslant i, j \leqslant n$. Define $x \tilde {<} y$ if $i < j$ or $i = j$ but $x <_i y$. Then $\tilde{\leqslant}$ is a linear order extending $\leqslant$, and $\mathcal{A}$ is still directed with respect to $\tilde{\leqslant}$. Therefore, without loss of generality we assume that $\leqslant$ is linear.

We prove this conclusion by induction on $| \Lambda |$, the size of $\Lambda$. It holds for $|\Lambda| = 1$ since $J = 0$. Now suppose that it holds for $|\Lambda| = s$ and consider $|\Lambda| = s+1$.

Let $\lambda$ be the minimal element in $\Lambda$ with respect to $\leqslant$. Therefore, $\mathcal{A} (e_{\mu}, e_{\lambda}) = e_{\lambda} A e_{\mu} = 0$ for all $\mu \neq \lambda \in \Lambda$. Let $\preccurlyeq$ be a linear order on $\Lambda$ such that $\lambda$ is the unique maximal element with respect to $\preccurlyeq$. Consider the quotient algebra $\bar{A} = A / A e_{\lambda} A$, which is standardly stratified for all linear orders on $\Lambda \setminus \{\lambda \}$ by the previous lemma. Thus the associated category $\bar {\mathcal{A}}$ of $\bar{A}$ is directed, and we can define $\bar{J}$ as well.

For $\mu \neq \lambda$, we have Hom$_A (P_{\lambda}, P_{\mu}) \cong e_{\lambda} A e_{\mu} = 0$ since $\lambda$ is minimal with respect to $\leqslant$ and $\mathcal{A}$ is directed. Therefore, tr$ _{P_{\lambda}} (P _{\mu}) = 0$, and the corresponding indecomposable projective $\bar{A}$-module $\bar{P}_{\mu}$ is isomorphic to $P_{\mu}$. Let $J_{\mu} = J e_{\mu}$. Then we have:
\begin{align*}
J_{\mu} & = \bigoplus _{\nu \neq \mu \in \Lambda} e_{\nu} A e_{\mu} = \bigoplus _{\nu > \mu} e_{\nu} A e_{\mu} \cong \bigoplus _{\nu > \mu} \text{Hom} _A (P_{\nu}, P_{\mu}) \\
& \cong \bigoplus _{\nu > \mu} \text{Hom} _{\bar{A}} (\bar{P}_{\nu}, \bar{P}_{\mu}) \cong \bigoplus _{\nu > \mu} e_{\nu} \bar{A} e_{\mu} = \bar{J}_{\mu}
\end{align*}
By the induction hypothesis, $J_{\mu} \cong \bar{J}_{\mu}$ is a projective $\bar {A}$-module. By our choice of $\lambda$, it is actually a projective $A$-module since $e_{\lambda}$ acts on $J_{\mu}$ as 0. Therefore, $J_{\mu}$ is a projective $A$ module.

Since $J = \bigoplus _{\nu \in \Lambda} J_{\nu}$ and we have proved that all $J_{\nu}$ are projective for $\nu \neq \lambda$, it remains to prove that $J_{\lambda}$ is a projective module. To achieve this, we take the element $\mu \in \Lambda$ which is minimal in $\Lambda \setminus \{\lambda\}$ with respect to $\leqslant$. Therefore, $\mathcal{A} (e_{\nu}, e_{\mu}) = e_{\mu} A e_{\nu} \neq 0$ only if $\nu = \mu$ or $\nu = \lambda$.

Now define another linear order $\preccurlyeq '$ on $\Lambda$ such that $\mu$ is the unique maximal element with respect to $\preccurlyeq'$. Similarly, for all $\nu \neq \mu, \lambda$, we have $\bar{P}_{\nu}' \cong P_{\nu}$, where $\bar{A}' = A / A e_{\mu} A$ and $\bar{P}_{\nu}' = \bar{A}' e_{\nu}$. As before, the associated category of $\bar{A}'$ is directed so we can define $\bar{J}'$. Moreover, we have $M = \text{tr} _{P_{\mu}} (P_{\lambda}) \subseteq J_{\lambda}$. Thus we get the following commutative diagram:
\begin{equation*}
\xymatrix{ & & 0 \ar[d] & 0 \ar[d] & \\
0 \ar[r] & M \ar[r] \ar@{=}[d] & J_{\lambda} \ar[r] \ar[d] & \bar{J} _{\lambda}' \ar[r] \ar[d] & 0 \\
0 \ar[r] & M \ar[r] & P_{\lambda} \ar[r] \ar[d] & \bar{P} _{\lambda}' \ar[r] \ar[d] & 0 \\
 & & e_{\lambda} A e_{\lambda} \ar[d] \ar@{=}[r] & e_{\lambda} A e_{\lambda} \ar[d] \ar[r] & 0 \\
 & & 0 & 0 &
}
\end{equation*}
where $\bar{J}_{\lambda}' = \bar{J'} e_{\lambda}$. By the induction hypothesis on the quotient algebra $\bar {A'}$, $\bar{J}_{\lambda}'$ is a projective $\bar {A'}$-module. Clearly, each indecomposable summands of $\bar{J}_{\lambda}'$ is isomorphic to a certain $\bar{P}_{\nu}'$ with $\nu \neq \lambda, \mu$. But $\bar{P}_{\nu}' \cong P_{\nu}$.  Therefore, $\bar {J}_{\lambda}'$ is actually a projective $A$-module, and the top sequence in the above diagram splits, i.e., $J_{\lambda} \cong M \oplus \bar{J}_{\lambda}'$. But $M = \text{tr} _{P_{\mu}} (P_{\lambda})$ is also a projective $A$-module (which is actually isomorphic to a direct sum of copies of $P_{\mu}$) since $A$ is standardly stratified with respect to $\preccurlyeq'$ and $\mu$ is maximal with respect to this order. Thus $J_{\lambda}$ is a projective $A$-module as well. The proof is completed.
\end{proof}

\begin{proposition}
Suppose that the associated category $\mathcal{A}$ of $A$ is directed. If $J$ is a projective $A$-module then for an arbitrary pair $\lambda, \mu \in \Lambda$, tr$_{P_{\lambda}} (P_{\mu}) \cong P_{\lambda} ^{m_{\mu}}$ for some $m_{\mu} \geqslant 0$.
\end{proposition}

\begin{proof}
Let $\leqslant$ be a linear order on $\Lambda$ with respect to which $\mathcal{A}$ is directed. We can index all orthogonal primitive idempotents: $e_n > e_{n-1} > \ldots > e_1$. Let $P_i = A e_i$. Suppose that $e_{\lambda} = e_s$ and $e_{\mu} = e_t$. If $s < t$, tr$_{P_s} (P_t) = 0$ and the conclusion is trivially true. For $s = t$, the conclusion holds as well. So we assume $s > t$ and conduct induction on the difference $d = s -t$. Since it has been proved for $d =0$, we suppose that the conclusion holds for all $d \leqslant l$.

Now suppose $d = s - t = l+1$. Let $E_t = \mathcal{A} (e_t, e_t) \cong P_t /J_t \cong e_t A e_t$, which can be viewed as an $A$-module. We have the following exact sequence:
\begin{equation*}
\xymatrix{0 \ar[r] & J_t \ar[r] & P_t \ar[r] & E_t \ar[r] & 0}.
\end{equation*}
Since $J$ is projective, so is $J_t = Je_t$. Moreover, since $J_t = \bigoplus _{m \neq t} e_m A e_t$ and $e_m A e_t = 0$ if $m \ngtr t$, we deduce that $J_t$ has no summand isomorphic to $P_m$ with $m < t $. Therefore, $J_t \cong \bigoplus _{t+1 \leqslant i \leqslant n} P_i^{m_i}$, where $m_i$ is the multiplicity of $P_i$ in $J_t$.

The above sequence induces an exact sequence
\begin{equation*}
\xymatrix{0 \ar[r] & \bigoplus _{t+1 \leqslant i \leqslant n} \text{tr} _{P_s} (P_i)^{m_i} \ar[r] & \text{tr} _{P_s} (P_t) \ar[r] & \text{tr} _{P_s} (E_t) \ar[r] & 0}.
\end{equation*}
Clearly, $\text{tr} _{P_s} (E_t) = 0$, so
\begin{equation*}
\bigoplus _{t+1 \leqslant i \leqslant n} \text{tr} _{P_s} (P_i)^{m_i} \cong \text{tr} _{P_s} (P_t).
\end{equation*}
But for each $t+1 \leqslant i \leqslant n$, we get $s - i < s - t = l+1$. Thus by the induction hypothesis, each $\text{tr} _{P_s} (P_i)$ is a projective module isomorphic to a direct sum of $P_s$. Therefore, $\text{tr} _{P_s} (P_t)$ is a direct sum of $P_s$. The conclusion follows from induction.
\end{proof}

The condition that $\mathcal{A}$ is directed is required in the previous propositions since otherwise $J$ might not be well defined.

\begin{proposition}
If for each pair $\lambda, \mu \in \Lambda$, tr$ _{P_{\lambda}} (P_{\mu})$ is a projective module, then the associated category $\mathcal{A}$ is a directed category. Moreover, $A$ is standardly stratified for all linear orders.
\end{proposition}

\begin{proof}
The first statement comes from the proof of Proposition 1.3. Actually, if $\mathcal{A}$ is not directed, in that proof we find some $P_{\lambda}$ and $P_{\mu}$ such that tr$_{P_{\lambda}} (P_{\mu}) \neq 0$ is not projective.

To prove the second statement, we use induction on $|\Lambda|$. It is clearly true if $|\Lambda| = 1$. Suppose that the conclusion holds for $|\Lambda| \leqslant l$ and suppose that $|\Lambda| = l+1$. Let $\preccurlyeq$ be an arbitrary linear order on $\Lambda$ and take a maximal element $\lambda$ with respect to this linear order. Since tr$ _{P_{\lambda}} (P_{\mu})$ is a projective $A$-module for all $\mu \in \Lambda$ by the given condition, it is enough to show that the quotient algebra $\bar {A} =  A /A e_{\lambda} A$ has the same property. That is, tr$ _{\bar{P_{\mu}}} (\bar{P_{\nu}})$ is a projective $\bar {A}$-module for all $\mu, \nu \in \Lambda \setminus \{ \lambda \}$, where $\bar{P_{\mu}} = \bar {A} e_{\mu}$ and $\bar{P_{\nu}}$ is defined similarly. Then the conclusion will follow from induction hypothesis.

Let $\leqslant$ be a linear order on $\Lambda$ with respect to which $\mathcal{A}$ is directed. It restriction on $\Lambda \setminus \{ \lambda \}$ gives a linear order with respect to which $\bar{\mathcal{A}}$, the associated category of $\bar{A}$, is directed. Consider tr$ _{\bar{P_{\mu}}} (\bar{P_{\nu}})$. If $\mu \leqslant \nu$, this trace is 0 or $\bar{P_{\nu}}$. Thus we assume that $\mu > \nu$. Since tr$ _{P_{\mu}} (P_{\nu})$ is projective, it is isomorphic to a direct sum of of $P_{\mu}$, and we get the following exact sequence:
\begin{equation*}
\xymatrix{0 \ar[r] & \text{tr} _{P_{\mu}} (P_{\nu}) \cong P_{\mu}^m \ar[r] & P_{\nu} \ar[r] & M \ar[r] & 0}
\end{equation*}
with $e_{\mu} M = 0$.

Since tr$ _{P_{\lambda}} (P_{\mu})$ and tr$ _{P_{\lambda}} (P_{\nu})$ are also isomorphic to direct sums of $P_{\lambda}$, by considering the traces of $P_{\lambda}$ in the modules in the above sequence, we get a commutative diagram as follows:
\begin{equation*}
\xymatrix{ & 0 \ar[d] & 0 \ar[d] & 0 \ar[d] & \\
0 \ar[r] & P_{\lambda}^s \cong \text{tr} _{P_{\lambda}} (P_{\mu}^m) \ar[r] \ar[d] & P_{\lambda}^t \cong \text{tr} _{P_{\lambda}} (P_{\nu}) \ar[r] \ar[d] & P_{\lambda}^{t-s} \cong \text{tr} _{P_{\lambda}} (M) \ar[r] \ar[d] & 0 \\
0 \ar[r] & P_{\mu}^m \ar[r] \ar[d] & P_{\nu} \ar[r] \ar[d] & M \ar[r] \ar[d] & 0 \\
0 \ar[r] & \bar{P_{\mu}}^m \ar[r] \ar[d] & \bar{P_{\nu}} \ar[r] \ar[d] & \bar{M} \ar[r] \ar[d] & 0 \\
 & 0 & 0 & 0 &.}
\end{equation*}
Since $e_{\mu} M =0$, we get $e_{\mu} \bar{M} =0$ as well. Thus from the bottom row we conclude that tr$ _{\bar {P_{\mu}}} (\bar {P_{\nu}}) \cong \bar{P_{\mu}}^m$ is projective. This finishes the proof.
\end{proof}

Now we can prove the first theorem.

\begin{proof}
$(1) \Rightarrow (2)$: by Propositions 1.3 and 1.6.

$(2) \Rightarrow (3)$: by Proposition 1.7.

$(3) \Rightarrow (1)$: by Proposition 1.8.

$(2) \Rightarrow (4)$: By the assumption, there is a linear order $\leqslant$ on $\Lambda$ such that $\mathcal{A}$ is directed with respect to it. By Proposition 1.4, we know that $_{\leqslant} \Delta \cong A_0 = \bigoplus _{\lambda \in \Lambda} e_{\lambda} A e_{\lambda}$. Thus the projective dimension of $_{\leqslant} \Delta$ is at most 1 since we have the exact sequence $0 \rightarrow J \rightarrow A \rightarrow A_0 \rightarrow 0$ and $J$ is projective. Therefore, every $M \in \mathcal{F} (_{\leqslant} \Delta)$ has projective dimension at most 1.

Note that $\mathcal{F} (_{\leqslant} \Delta)$ is closed under cokernels of monomorphisms by Proposition 1.4. Let $\preccurlyeq$ be an arbitrary linear order on $\Lambda$. By the second statement of Theorem 0.3 (which will be proved later), $\mathcal{F} (_{\preccurlyeq} \Delta) \subseteq \mathcal{F} (_{\leqslant} \Delta)$. Thus every $M \in \mathcal{F} (_{\preccurlyeq} \Delta) \subseteq \mathcal{F} (_{\leqslant} \Delta)$ has projective dimension at most 1.

$(4) \Rightarrow (3)$: Take an arbitrary pair $\lambda, \mu \in \Lambda$. We want to show that tr$ _{P_{\lambda}} (P_{\mu})$ is a projective $A$-module. Clearly, there exists a linear order $\preccurlyeq$ on $\Lambda$ such that $\lambda$ is the maximal element in $\Lambda$ and $\mu$ is the maximal element in $\Lambda \setminus \{ \lambda \}$. Therefore, by the definition, we have a short exact sequence
\begin{equation*}
\xymatrix{ 0 \ar[r] & \text{tr} _{P_{\lambda}} (P_{\mu}) \ar[r] & P_{\mu} \ar[r] & _{\preccurlyeq} \Delta_{\mu} \ar[r] & 0}.
\end{equation*}
Clearly, $_{\preccurlyeq} \Delta_{\mu} \in \mathcal{F} (_{\preccurlyeq} \Delta)$. Therefore, it has projective dimension at most 1, which means that $\text{tr} _{P_{\lambda}} (P_{\mu})$ is projective.
\end{proof}

This immediately gives us the following characterization of algebras properly stratified for all linear orders.

\begin{corollary}
Let $A, \mathcal{A}$ and $J$ be as in the previous theorem. Then $A$ is properly stratified for all linear orders if and only if $\mathcal{A}$ is a directed category and $J$ is a left and right projective $A$-module.
\end{corollary}

\begin{proof}
By Proposition 1.2 and Theorem 0.1.
\end{proof}

The following example describes an algebra which is standardly stratified (but not properly stratified) for all linear orders.

\begin{example}
Let $A$ be the algebra defined by the following quiver with relations: $\delta^2 = \beta \delta =0$, $\alpha \delta = \gamma \beta$.
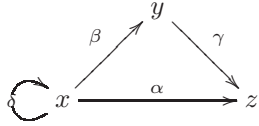
\begin{equation*}
\xymatrix { & y \ar[dr] ^{\gamma} & \\
x \ar@(ld, lu)|[] {\delta} \ar[ur] ^{\beta} \ar[rr] ^{\alpha} & & z}
\end{equation*}
The indecomposable left projective modules are described as follows:
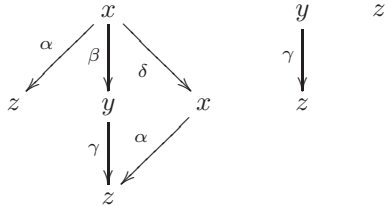
\begin{equation*}
\xymatrix{ & x \ar[dl] _{\alpha} \ar[dr]_{\delta} \ar[d]_{\beta} & \\
z & y \ar[d]_{\gamma} & x \ar[dl] _{\alpha} \\
 & z & }
\qquad \xymatrix{ y \ar[d] _{\gamma} \\ z}
\qquad {z}
\end{equation*}
It is easy to see that $J \cong P_y \oplus P_z \oplus P_z$ is a left projective $A$-module, so $A$ is standardly stratified for all linear orders.

On the other hand, the indecomposable right projective modules have the following structures:
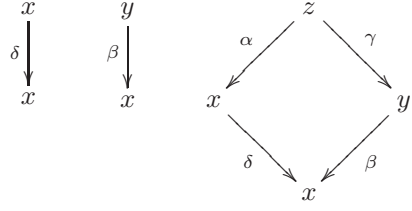
\begin{equation*}
\xymatrix{ x \ar[d] _{\delta} \\ x}
\qquad \xymatrix{ y \ar[d] _{\beta} \\ x}
\qquad \xymatrix{ & z \ar[dl] _{\alpha} \ar[dr] ^{\gamma} & \\
x \ar[dr] _{\delta} &  & y \ar[dl] ^{\beta} \\
 & x & }
\end{equation*}
Thus $J$ is not a right projective $A$-module, and hence $A$ is not properly stratified for all linear orders.
\end{example}

\section{The classification}

Throughout this section let $A$ be a finite-dimensional basic algebra with a chosen set of orthogonal primitive idempotents $\{ e _{\lambda} \} _{\lambda \in \Lambda}$ such that $\sum _{\lambda \in \Lambda} e_{\lambda} = 1$. We also suppose that the associated category $\mathcal{A}$ is directed. Thus we can define $A_0$ and $J$ as before, and consider its associated graded algebra $\check{A} = \bigoplus _{i \geqslant 0} J^i / J^{i+1}$, where we set $J^0 = A$. Correspondingly, for a finitely generated $A$-module $M$, its associated graded $\check{A}$-module $\check{M} = \bigoplus _{i \geqslant 0} J^iM/ J^{i+1} M$. Clearly, we have $\check{A}_i \cdot \check{A}_j = \check{A} _{i+j}$ for $i, j \geqslant 0$.

\begin{lemma}
Let $M$ be an $A$-module and $\check{M}$ be the associated graded $\check{A}$-module. Then $\check{M}$ is a graded projective $\check{A}$-module if and only if $M$ is a projective $A$-module.
\end{lemma}

\begin{proof}
Without loss of generality we can assume that $M$ is indecomposable. Since the free module $_AA$ is sent to the graded free module $_{\check{A}} \check{A}$ by the grading process, we know that $\check{M}$ is a graded projective $\check{A}$-module generated in degree 0 if $M$ is a projective $A$-module. Now suppose that $M$ is not a projective $A$-module but $\check{M}$ is a projective $\check{A}$-module, and we want to get a contradiction.

Let $p: P \rightarrow M$ be a projective covering map of $M$. Then $\dim _{\Bbbk} M < \dim _{\Bbbk} P$ since $M$ is not projective. Moreover, $\top P = P / \rad P \cong \top M = M / \rad M$. The surjective map $p$ gives a graded surjective map $\check{p}: \check{P} \rightarrow \check{M}$. Since $\check{M}$ is supposed to be projective and $\check{P}$ is projective, we get a splitting exact sequence of graded projective $\check{A}$-modules generated in degree 0 as follows:
\begin{equation*}
\xymatrix{0 \ar[r] & \check{L} \ar[r] & \check{P} \ar[r] & \check{M} \ar[r] & 0}.
\end{equation*}
Notice that $\check{L} \neq 0$ since $\dim_{\Bbbk} \check{M} = \dim_{\Bbbk} M < \dim_{\Bbbk} P = \dim_{\Bbbk} \check{P}$. Consider the degree 0 parts. We have a splitting sequence of $A_0$-modules
\begin{equation*}
\xymatrix{0 \ar[r] & \check{L}_0 \ar[r] & \check{P}_0 \ar[r] & \check{M}_0 \ar[r] & 0}
\end{equation*}
which by definition is isomorphic to
\begin{equation*}
\xymatrix{0 \ar[r] & \check{L}_0 \ar[r] & P/JP \ar[r] & M/JM \ar[r] & 0}.
\end{equation*}
View them as $A$-modules on which $J$ acts as 0. Observe that $J$ is contained in the radical of $A$. Thus $\top P \cong \top (P/JP)$ and $\top M \cong \top (M/JM)$. But the above sequence splits, hence $\top P \cong \top \check{L}_0 \oplus \top M$, contradicting the fact that $\top P \cong \top M$. This finishes the proof.
\end{proof}

The above lemma still holds if we replace left modules by right modules. It immediately implies the following result, which is a part of Theorem 0.2:

\begin{proposition}
Let $A$ be a finite-dimensional basic algebra whose associated category is directed. Let $\check{A}$ be the associated graded algebra. Then $A$ is standardly (resp., properly) stratified for all linear orders if and only if so is $\check{A}$.
\end{proposition}

\begin{proof}
The algebra $A$ is standardly stratified for all linear orders if and only if the associated category $\mathcal{A}$ is a directed category and $J$ is a projective $A$-module. This happens if and only if the graded category $\check {\mathcal{A}}$ is a directed category and $\check{J}$ is a projective $\check{A}$-module by the previous lemma, and if and only if $\check{A}$ is standardly stratified for all linear orders.
\end{proof}

In general $\check{A}$ (when viewed as a non-graded algebra) is not isomorphic to $A$, as shown by the following example.

\begin{example}
Let $A$ be the algebra described in Example 1.10, where we proved that it is standardly stratified for all linear orders. It is easy to see $J/J^2 = \langle \bar{\alpha}, \bar{\beta}, \bar{\gamma} \rangle$, $J^2 / J^3 = \langle \bar{\gamma} \bar{\beta} \rangle$ and $J^3 = 0$. Therefore, $\check{A}$ is defined by the following quiver with relations $\bar{\alpha} \delta = \bar{\beta} \bar{\delta} = 0$, which is not isomorphic to $A$.
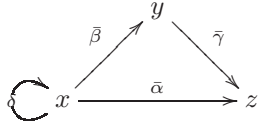
\begin{equation*}
\xymatrix { & y \ar[dr] ^{\bar{\gamma}} & \\
x \ar@(ld, lu)|[] {\delta} \ar[ur] ^{\bar{\beta}} \ar[rr] ^{\bar{\alpha}} & & z}
\end{equation*}
The indecomposable left projective $\check{A}$-modules have the following structures:
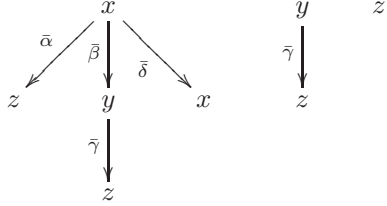
\begin{equation*}
\xymatrix{ & x \ar[dr] _{\bar{\delta}} \ar[dl] _{\bar{\alpha}} \ar[d] _{\bar{\beta}} & \\
z & y \ar[d] _{\bar{\gamma}} & x \\
 & z & }
\qquad \xymatrix{ y \ar[d] _{\bar{\gamma}} \\ z}
\qquad {z}
\end{equation*}
It still holds that $\check{J} \cong P_y \oplus P_z \oplus P_z$, so $\check{A}$ is standardly stratified for all linear orders.

The indecomposable right projective $\check{A}$-modules are as follows:
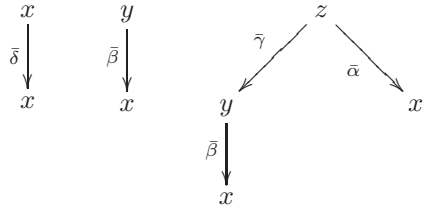
\begin{equation*}
\xymatrix{ x \ar[d] _{\bar{\delta}} \\ x} \qquad \xymatrix{ y \ar[d] _{\bar{\beta}} \\ x} \qquad \xymatrix{ & z \ar[dl] _{\bar{\gamma}} \ar[dr] _{\bar{\alpha}} & \\ y \ar[d] _{\bar{\beta}} & & x \\ x}
\end{equation*}
We deduce that $\check{A}$ is not properly stratified for all linear orders since $\check{J}$ is not a right projective module.
\end{example}

Let $X$ be an $(A_0, A_0)$-bimodule. We denote the tensor algebra generated by $A_0$ and $X$ to be $A_0[X]$. That is, $A_0[X] = A_0 \oplus X \oplus (X \otimes _{A_0} X) \oplus \ldots$. With this notation, we have:

\begin{lemma}
Let $A = \bigoplus _{i \geqslant 1} A_i$ be a finite-dimensional graded algebra with $A_i \cdot A_j = A_{i+j}$, $i, j \geqslant 0$. Then $J = \bigoplus _{i \geqslant 1} A_i$ is a projective $A$-module if and only if $A \cong A_0 [A_1]$, and $A_1$ is a projective $A_0$-module.
\end{lemma}

\begin{proof}
Suppose that $A = A_0[A_1]$ and $A_1$ is a projective $A_0$-module. Observe that $A$ is a finite-dimensional algebra, so there exists a minimal number $n \geqslant 0$ such that $A_{n+1} \cong \underbrace{ A_1 \otimes _{A_0} \ldots \otimes _{A_0} A_1}_{n+1} = 0$. Without loss of generality we assume that $n >0$ since otherwise $J = 0$ is a trivial projective module. Therefore,
\begin{align*}
J & = \bigoplus _{i = 1}^n A_i = A_1 \oplus (A_1 \otimes _{A_0} A_1) \oplus  \ldots \oplus (\underbrace{ A_1 \otimes _{A_0} \ldots \otimes _{A_0} A_1}_n) \\
& = A_1 \oplus (A_1 \otimes _{A_0} A_1) \oplus  \ldots \oplus (\underbrace{ A_1 \otimes _{A_0} \ldots \otimes _{A_0} A_1}_{n+1}) \\
& \cong \big{(} A_0 \oplus A_1 \oplus  \ldots \oplus (\underbrace{ A_1 \otimes _{A_0} \ldots \otimes _{A_0} A_1}_n) \big{)} \otimes _{A_0} A_1 \\
& = A \otimes _{A_0} A_1,
\end{align*}
which is projective since $A_1$ is a projective $A_0$-module.

Conversely, suppose that $J = \bigoplus _{i \geqslant 1} A_i$ is a projective $A$-module. Since $J$ is generated in degree 1, $A_1$ must be a projective $A_0$-module. Moreover, we have a minimal projective resolution of $A_0$ as follows
\begin{equation*}
\xymatrix {\ldots \ar[r] & A \otimes _{A_0} A_1 \ar[r] & A \ar[r] & A_0 \ar[r] & 0}.
\end{equation*}
But we also have the short exact sequence $0 \rightarrow J \rightarrow A \rightarrow A_0 \rightarrow 0$. Since $J$ is projective, we deduce that $J \cong A \otimes _{A_0} A_1$. Therefore, $A_2 \cong A_1 \otimes _{A_0} A_1$, $A_3 \cong A_1 \otimes _{A_0} A_1 \otimes _{A_0} A_1$, and so on. Thus $A \cong A_0[A_1]$ as claimed.
\end{proof}

Now we prove Theorem 0.2.

\begin{proof}
For standardly stratified algebras, the equivalence of (1) and (2) has been established in Proposition 2.2, and the equivalence of (2) and (3) comes from the previous lemma and Theorem 0.1. Since all arguments work for right modules, we also have the equivalence of these statements for properly stratified algebras.
\end{proof}

We end this section with a combinatorial description of $\check {\mathcal{A}}$, the associated graded category of $\check{A}$ stratified for all linear orders. Let $Q = (Q_0, Q_1)$ be a finite acyclic quiver, where both the vertex set $Q_0$ and the arrow set $Q_1$ are finite. We then define a \textit{quiver of bimodules} $\tilde{Q} = (Q_0, Q_1, f, g)$: to each vertex $v \in Q_0$ the map $f$ assigns a finite-dimensional local algebra $A_v$, i.e., $f(v) = A_v$; for each arrow $\alpha: v \rightarrow w$, $g(\alpha)$ is a finite-dimensional $(A_w, A_v)$-bimodule.

The quiver of bimodules $\tilde{Q}$ determines a category $\mathcal{C}$. Explicitly, $\Ob \mathcal{C} = Q_0$. The morphisms between an arbitrary pair of objects $v, w \in Q_0$ are defined as follows. Let
\begin{equation*}
\xymatrix {\gamma: v = v_0 \ar[r] ^{\alpha_1} \ar[r] & v_1 \ar[r] ^{\alpha_2} & \ldots \ar[r] ^{\alpha_{n-1}} & v_{n-1} \ar[r] ^{\alpha_n} & v_n = w}
\end{equation*}
be an oriented path in $Q$. We define
\begin{equation*}
M_{\gamma} = g(\alpha_n) \otimes _{f (v_{n-1})} g(\alpha_{n-1}) \otimes _{f (v_{n-2})} \ldots \otimes_{f (v_1)} g(\alpha_1).
\end{equation*}
This is a $(f(w), f(v))$-bimodule. Then
\begin{equation*}
\mathcal{C} (v, w) = \bigoplus _{\gamma \in P(v, w)} M_{\gamma},
\end{equation*}
Where $P(v, w)$ is the set of all oriented paths from $v$ to $w$. The composite of morphisms is defined by tensor product. We call a category defined in this way a \textit{free directed} category. It is \textit{left regular} if for every arrow $\alpha: v \rightarrow w$, $g(\alpha)$ is a left projective $f(w)$-module. Similarly, we define \textit{right regular} categories. It is \textit{regular} if this category is both left regular and right regular.

Using these terminologies, we get the following combinatorial description of $\check{\mathcal{A}}$.

\begin{theorem}
Let $A$ be a finite-dimensional basic algebra whose associated category is directed. Then $A$ is standardly (resp., properly) stratified for all linear orders if and only if the graded category $\check {\mathcal{A}}$ is a left regular (resp., regular) free directed category.
\end{theorem}

\begin{proof}
It is straightforward to check that if $\check {\mathcal{A}}$ is a left regular free directed category, then $\check{A}$ satisfies (3) in Theorem 0.2. So $A$ is standardly stratified for all linear orders. Conversely, if $A$ is standardly stratified for all linear orders, then $\check {\mathcal{A}}$ is a directed category, and $\check{A}$ is a tensor algebra generated by $A_0 = \bigoplus _{\lambda \in \Lambda} e_{\lambda} A e_{\lambda}$ and a projective left $A_0$-module $\check{A}_1$. Define $Q = (Q_0, Q_1, f, g)$ in the following way: $Q_0 = \Lambda$, and $f(\lambda) = e_{\lambda} A_0 e_{\lambda}$. Arrows and the map $g$ are defined as follows: for $\lambda \neq \mu \in Q_0$, we put an arrow $\phi: \lambda \rightarrow \mu$ if $e_{\mu} \check{A}_1 e_{\lambda} \neq 0$ and define $g(\phi) = e_{\mu} \check{A}_1 e_{\lambda}$. In this way $Q$ defines a left regular free directed category which is isomorphic to $\check {\mathcal{A}}$. Since all arguments work for right modules, the proof is completed.
\end{proof}

\section{ Whether $\mathcal{F} (_{\preccurlyeq} \Delta)$ is closed under cokernels of monomorphisms?}

In Proposition 1.4 we proved that if $\mathcal{A}$ is directed and standardly stratified with respect to a linear order $\leqslant$, then the corresponding category $\mathcal{F} (_{\leqslant} \Delta)$ is closed under cokernels of monomorphisms. This result motivates us to study the general situation. Suppose $A$ is standardly stratified with respect to a fixed linear ordered set $(\Lambda, \preccurlyeq)$. In the following lemmas we describe several equivalent conditions for $\mathcal{F} (_{\preccurlyeq} \Delta)$ to be closed under cokernels of monomorphisms.

\begin{lemma}
The category $\mathcal{F} (_{\preccurlyeq} \Delta)$ is closed under cokernels of monomorphisms if and only if for each exact sequence $0 \rightarrow L \rightarrow M \rightarrow N \rightarrow 0$, where $L, M \in \mathcal{F} (_{\preccurlyeq} \Delta)$ and $L$ is indecomposable, $N$ is also contained in $\mathcal{F} (_{\preccurlyeq} \Delta)$.
\end{lemma}

\begin{proof}
The only if direction is trivial, we prove the other direction: for each exact sequence $0 \rightarrow \tilde{L} \rightarrow M \rightarrow N \rightarrow 0$ with $\tilde{L}, M \in \mathcal{F} (_{\preccurlyeq} \Delta)$, we have $N \in \mathcal{F} (_{\preccurlyeq} \Delta)$ as well.

We use induction on the number of indecomposable direct summands of $\tilde{L}$. If $\tilde{L}$ is indecomposable, the conclusion holds obviously. Now suppose that the if part is true for $\tilde{L}$ with at most $l$ indecomposable summands. Assume that $\tilde{L}$ has $l+1$ indecomposable summands. Taking an indecomposable summand $L_1$ of $\tilde{L}$ we have the following diagram:
\begin{equation*}
\xymatrix{ & 0 \ar[d] & 0 \ar[d] \\
0 \ar[r] & L_1 \ar[d] \ar@{=}[r] & L_1 \ar[r] \ar[d] & 0\\
0 \ar[r] & \tilde{L} \ar[r] \ar[d] & M \ar[r] \ar[d] & N \ar[r] \ar@{=}[d] & 0 \\
0 \ar[r] & \bar{L} \ar[d] \ar[r] & \bar{M} \ar[d] \ar[r] & N \ar[r] & 0\\
 & 0 & 0}
\end{equation*}
Considering the middle column, $\bar{M} \in \mathcal{F} (_{\preccurlyeq} \Delta)$ by the given condition. Therefore, we conclude that $N \in \mathcal{F} (_{\preccurlyeq} \Delta)$ by using the induction hypothesis on the bottom row.
\end{proof}

\begin{lemma}
The category $\mathcal{F} (_{\preccurlyeq} \Delta)$ is closed under cokernels of monomorphisms if and only if for each exact sequence $0 \rightarrow L \rightarrow P \rightarrow N \rightarrow 0$, where $L, P \in \mathcal{F} (_{\preccurlyeq} \Delta)$ and $P$ is projective, $N$ is also contained in $\mathcal{F} (_{\preccurlyeq} \Delta)$.
\end{lemma}

\begin{proof}
It suffices to show the if part. Let $0 \rightarrow L \rightarrow M \rightarrow N \rightarrow 0$ be an exact sequence with $L, M \in \mathcal{F} (_{\preccurlyeq} \Delta)$. We want to show $N \in \mathcal{F} (_{\preccurlyeq} \Delta)$ as well. Let $P$ be a projective cover of $M$. Then we have a commutative diagram by the Snake Lemma:
\begin{equation*}
\xymatrix{ & 0 \ar[r] & \Omega(M) \ar[r] \ar[d] & N' \ar[r] \ar[d] & L \ar[r] & 0\\
 & & P \ar@{=}[r] \ar[d] & P \ar[d] \\
0 \ar[r] & L \ar[r] & M \ar[r] & N \ar[r] & 0,}
\end{equation*}
where all rows and columns are exact. Clearly, $\Omega(M) \in \mathcal{F} (_{\preccurlyeq} \Delta)$, so is $N'$ since $\mathcal{F} (_{\preccurlyeq} \Delta)$ is closed under extension. Considering the last column, we conclude that $N \in \mathcal{F} (_{\preccurlyeq} \Delta)$ by the given condition.
\end{proof}

Every $M \in \mathcal{F} (_{\preccurlyeq} \Delta)$ has a $_{\preccurlyeq} \Delta$-filtration $\xi$ and we define $[M: \Delta_{\lambda}]$ to be the number of factors isomorphic to $\Delta_{\lambda}$ in $\xi$. This number is independent of the particular $\xi$ (see \cite{Dlab1} and \cite{Erdmann}). We then define $l(M) = \sum _{\lambda \in \Lambda} [M: \Delta_{\lambda}]$ and call it the \textit{filtration length} of $M$, which is also independent of the choice of $\xi$.

\begin{lemma}
The category $\mathcal{F} (_{\preccurlyeq} \Delta)$ is closed under cokernels of monomorphisms if and only if for each exact sequence $0 \rightarrow \Delta_{\lambda} \rightarrow M \rightarrow N \rightarrow 0$, where $M \in \mathcal{F} (_{\preccurlyeq} \Delta)$ and $\lambda \in \Lambda$, $N$ is also contained in $\mathcal{F} (_{\preccurlyeq} \Delta)$.
\end{lemma}

\begin{proof}
We only need to show the if part. Let $0 \rightarrow L \rightarrow M \rightarrow N \rightarrow 0$ be an exact sequence with $L, M \in \mathcal{F} (_{\preccurlyeq} \Delta)$. We use induction on the filtration length of $L$.

If $l(L) = 1$, then $L \cong \Delta_{\lambda}$ for some $\lambda \in \Lambda$ and the conclusion holds clearly. Suppose that the conclusion is true for all objects in $\mathcal{F} (_{\preccurlyeq} \Delta)$ with filtration length at most $s$ and assume $l(L) = s+1$. Then we have an exact sequence $0 \rightarrow L' \rightarrow L \rightarrow \Delta_{\lambda} \rightarrow 0$ for some $\lambda \in \Lambda$ and $l(L') = s$. Then we have a commutative diagram by the Snake Lemma:

\begin{equation*}
\xymatrix{ & 0 \ar[r] & L' \ar[d] \ar[r] & L \ar[r] \ar[d] & \Delta_{\lambda} \ar[r] & 0\\
 & & M \ar[d] \ar@{=}[r] & M \ar[d]\\
0 \ar[r] & \Delta_{\lambda} \ar[r] & N' \ar[r] & N \ar[r] & 0.}
\end{equation*}

Consider the first column. By induction hypothesis, $N' \in \mathcal{F} (_{\preccurlyeq} \Delta)$. By considering the bottom row we conclude that $N \in \mathcal{F} (_{\preccurlyeq} \Delta)$ from the given condition.
\end{proof}

Now we prove Theorem 0.3.

\begin{proof}
The first statement follows immediately from the above lemmas.

Now we prove the second statement. The claim is clear if $|\Lambda|=1$. So we assume $\Lambda$ has more than one elements. Take an element $\lambda \in \Lambda$ maximal with respect to $\preccurlyeq '$. Clearly, $\Delta'_{\lambda} \cong P_{\lambda} \in \mathcal{F} (_{\preccurlyeq} \Delta)$. Consider the quotient algebra $\bar{A} = A / Ae_{\lambda}A$. It is standardly stratified with respect to the restricted linear order $\preccurlyeq '$ on $\Lambda \setminus \{ \lambda\}$. We claim $\bar{A} \in \mathcal{F} (_{\preccurlyeq} \Delta)$ as well. Indeed, consider the exact sequence
\begin{equation*}
\xymatrix { 0 \ar[r] & Ae_{\lambda}A \ar[r] & A \ar[r] & \bar{A} \ar[r] & 0}.
\end{equation*}
Note that $A e_{\lambda}A \in \mathcal{F} (_{\preccurlyeq} \Delta)$ as a direct sum of $P_{\lambda}$. Therefore, $\bar{A} \in \mathcal{F} (_{\preccurlyeq} \Delta)$ as well since $\mathcal{F} (_{\preccurlyeq} \Delta)$ is closed under cokernels of monomorphisms. Taking a maximal element $\nu \in \Lambda \setminus \{\lambda \}$ with respect to $\preccurlyeq'$ and repeating the above procedure, we conclude that $_{\preccurlyeq'} \Delta_{\nu} \cong \bar{A} e_{\nu} \in \mathcal{F} (_{\preccurlyeq} \Delta)$ and $\bar{\bar{A}} = \bar{A} / \bar{A} e_{\nu} \bar{A} \in \mathcal{F} (_{\preccurlyeq} \Delta)$. Recursively, we proved that $_{\preccurlyeq'} \Delta_{\lambda} \in \mathcal{F} (\Delta)$ for all $\lambda \in \Lambda$, so $\mathcal{F} (_{\preccurlyeq'} \Delta) \subseteq \mathcal{F} (_{\preccurlyeq} \Delta)$.

The third statement can be proved by induction on $|\Lambda|$ as well. If $|\Lambda| = 1$, the claim is clear since $A \cong \Bbbk$. Suppose that the conclusion holds for $|\Lambda| = s$ and let $\Lambda$ be a linear ordered set with $s+1$ elements. Take a maximal element $\lambda$ in $\Lambda$. Then $\bar{A} = A / Ae_{\lambda}A$ is also a quasi-hereditary algebra. Moreover, $\mathcal{F} (_{\bar{A}} \Delta)$ is still closed under cokernels of monomorphisms. By the induction hypothesis, $\bar{A}$ is the quotient of a finite-dimensional hereditary algebra, and all standard modules are simple. Note that these standard $\bar {A}$-modules can be viewed as standard $A$-modules.

Choose a composition series of $\Delta_{\lambda} \cong P_{\lambda}$: $0 = M_0 \subseteq M_1 \subseteq \ldots \subseteq M_t = \Delta_{\lambda}$. It is clear that $M_i / M_{i-1} \cong S_{\lambda}$ if and only if $i=t$ since $A$ is quasi-hereditary. But $\mathcal{F} (_{\preccurlyeq} \Delta)$ is closed under cokernels of monomorphisms and $\Delta_{\mu} \cong S_{\mu} \cong P_{\mu} / \rad P_{\mu}$ for all $\mu \in \Lambda \setminus \{ \lambda \}$, we deduce that $S_{\lambda} \in \mathcal{F} (_{\preccurlyeq} \Delta)$. Thus $S_{\lambda} \cong \Delta_{\lambda}$.

It remains to show that the ordinary quiver of $A$ has no oriented cycles. For $P_{\lambda} = Ae_{\lambda}$ and $P_{\mu} = Ae_{\mu}$ with $\lambda \nsucceq \mu$, we claim $e_{\lambda} A e_{\mu}  \cong \text{Hom} _{\mathcal{A}} (P_{\lambda}, P_{\mu}) = 0$. Indeed, since $P_{\mu} \in \mathcal{F} (_{\preccurlyeq} \Delta)$ and all standard modules are simple, the composition factors of $P_{\mu}$ are those $S_{\nu}$ with $\nu \in \Lambda$ and $\nu \succcurlyeq \mu$. Thus $P_{\mu}$ has no composition factors isomorphic to $S_{\lambda}$. Therefore, Hom$ _{\mathcal{A}} (P_{\lambda}, P_{\mu}) =0$.
\end{proof}

An immediate corollary is:

\begin{corollary}
If $A$ is standardly stratified with respect to two different linear orders $\preccurlyeq$ and $\preccurlyeq'$ such that both $\mathcal{F} (_{\preccurlyeq} \Delta)$ and $\mathcal{F} (_{\preccurlyeq} \Delta)$ are closed under cokernels of monomorphisms, then $\mathcal{F} (_{\preccurlyeq} \Delta) = \mathcal{F} (_{\preccurlyeq'} \Delta)$.
\end{corollary}

\begin{proof}
It is straightforward from the second statement of the previous theorem.
\end{proof}

\section{Examples}

In this section we describe an algorithm to determine whether there is a linear order $\preccurlyeq$ with respect to which $A$ is standardly stratified and $\mathcal{F} (_{\preccurlyeq} \Delta)$ is closed under cokernels, as well as several examples.

\subsection{An algorithm.}

Given an arbitrary algebra $A$, we want to check whether there exists a linear order $\preccurlyeq$ for which $A$ is standardly stratified and the corresponding category $\mathcal{F} (_{\preccurlyeq} \Delta)$ is closed under cokernels. It is certainly not an ideal way to check all linear orders. In the rest of this paper we will describe an algorithm to construct a set $\mathcal{L}$ of linear orders for $A$ satisfying the following property: $A$ is standardly stratified for every linear order in $\mathcal{L}$; moreover, if there is a linear order $\preccurlyeq$ for which $A$ is standardly stratified and $\mathcal{F} (_{\preccurlyeq} \Delta)$ is closed under cokernels, then $\preccurlyeq \in \mathcal{L}$.

As before, choose a set $\{e_{\lambda} \} _{\lambda \in \Lambda}$ of orthogonal primitive idempotents in $A$ such that $\sum _{\lambda \in \Lambda} e_{\lambda} = 1$ and let $P_{\lambda} = A e_{\lambda}$. The algorithm is as follows:

\begin{enumerate}
\item Define $O_1 = \{ \lambda \in \Lambda \mid \forall \mu \in \Lambda, \text{tr} _{P_{\lambda}} (P_{\mu}) \cong P_{\lambda}^{m_{\mu}} \}$. If $O_1 = \emptyset$, the algorithm ends at this step. Otherwise, continue to the second step.
\item Define a partial order $\leqslant '$ on $O_1$: $\lambda \leqslant ' \mu$ if and only if tr$_{P_{\mu}} (P_{\lambda}) \neq 0$ for $\lambda, \mu \in O_1$. We can check that this partial order $\leqslant'$ is well defined.
\item Take $e_{s_1} \in O_1$ which is maximal with respect to $\leqslant'$. Let $\bar {A} = A / A e_{s_1} A$.
\item Repeat the above steps for $\bar{A}$ recursively until the algorithm ends. Thus we get a chain of $t$ idempotents $e_{s_1}, e_{s_2}, \ldots, e_{s_t}$ and define $e_{s_1} \succ e_{s_2} \succ \ldots \succ e_{s_t}$.
\end{enumerate}
Let $\tilde{ \mathcal{L} }$ be the set of all linear orders obtained from the above algorithm, and let $\mathcal{L} \subseteq \tilde{ \mathcal{L} }$ be the set of linear orders with length $n = |\Lambda|$.

\begin{proposition}
The algebra $\mathcal{A}$ is standardly stratified for every $\preccurlyeq \in \mathcal{L}$.
\end{proposition}

\begin{proof}
We use induction on $|\Lambda|$. The conclusion is clearly true if $|\Lambda| = 1$. Suppose that it holds for $|\Lambda| \leqslant n$ and consider the case that $|\Lambda| = n+1$.

Let $\preccurlyeq$ be an arbitrary linear order in $\mathcal{L}$ and take the unique maximal element $\lambda \in \Lambda$ with respect to $\preccurlyeq$. Consider the quotient algebra $\bar{A} = A / A e_{\lambda} A$. Then $\bar{A}$, by the induction hypothesis and our algorithm, is standardly stratified with respect to the restricted order on $\Lambda \setminus \{ \lambda \}$. It is clear from our definition of $\preccurlyeq$ that $A e_{\lambda} A = \bigoplus _{\mu \in \Lambda} \text{tr} _{P_{\lambda}} (P_{\mu})$ is projective. Thus $A$ is standardly stratified for $\preccurlyeq$.
\end{proof}

The next proposition tells us that it is enough to check linear orders in $\mathcal{L}$ to determine whether there exists a linear order $\preccurlyeq$ for which $\mathcal{A}$ is standardly stratified and the corresponding category $\mathcal{F} (_{\preccurlyeq} \Delta)$ is closed under cokernels of monomorphisms.

\begin{proposition}
Let $\preccurlyeq$ be a linear order on $\Lambda$ such that $A$ is standardly stratified and the corresponding category $\mathcal{F} (_{\preccurlyeq} \Delta)$ is closed under cokernels of monomorphisms. Then $\preccurlyeq \in \mathcal{L}$.
\end{proposition}

\begin{proof}
The proof relies on induction on $|\Lambda|$. The claim is clear if $|\Lambda| = 1$. Suppose that the conclusion is true for $| \Lambda | \leqslant l$ and let $\Lambda$ be a set with $n=l+1$ elements. Note that our algorithm is defined recursively. Furthermore, the quotient algebra $\bar {A} = A / A e_{\lambda} A$ is also standardly stratified for the restricted linear order $\preccurlyeq$ on $\Lambda \setminus \{ \lambda \}$, where $\lambda$ is maximal with respect to $\preccurlyeq$; $\bar{\Delta} = \bigoplus _{\mu \in \Lambda \setminus \{\lambda \} } \Delta_{\mu}$; and $\mathcal{F} (\bar {\Delta})$ is also closed under cokernels of monomorphisms. Thus by induction it suffices to show that $\lambda \in O_1$, and is maximal in $O_1$ with respect to $\leqslant'$ (see the second step of the algorithm).

Consider $P_{\lambda} \cong \Delta_{\lambda}$. Since $A$ is standardly stratified for $\preccurlyeq$,  for each $\mu \in \Lambda$, tr$_{P_{\lambda}} (P_{\mu})$ is a projective module. Thus $\lambda \in O_1$.

If $\lambda$ is not maximal in $O_1$ with respect to $\leqslant '$, then we can choose some $\mu \in O_1$ such that $\mu >' \lambda$, i.e., tr$ _{P_{\mu}} (P_{\lambda}) \neq 0$. Since $\mu \in O_1$, by definition, tr$ _{P_{\mu}} (P_{\lambda}) \cong P_{\mu} ^m$ for some $m \geqslant 1$. Now consider the exact sequence:
\begin{equation*}
\xymatrix{ 0 \ar[r] & \text{tr} _{P_{\mu}} (P_{\lambda}) \ar[r] & P_{\lambda} \ar[r] & P_{\lambda} / \text{tr} _{P_{\mu}} (P_{\lambda}) \ar[r] & 0.}
\end{equation*}
Since tr$ _{P_{\mu}} (P_{\lambda}) \cong P_{\mu} ^m \in \mathcal{F} (_{\preccurlyeq} \Delta)$, $P_{\lambda} \in \mathcal{F} (_{\preccurlyeq} \Delta)$, and $\mathcal{F} (_{\preccurlyeq} \Delta)$ is closed under cokernels of monomorphism, we conclude that $P_{\lambda} / \text{tr} _{P_{\mu}} (P_{\lambda}) \in \mathcal{F} (_{\preccurlyeq} \Delta)$. This is impossible. Indeed, since $P_{\lambda} / \text{tr} _{P_{\mu}} (P_{\lambda})$ has a simple top $S_{\lambda} \cong P_{\lambda} / \rad P_{\lambda}$, if it is contained in $\mathcal{F} (_{\preccurlyeq} \Delta)$, then it has a filtration factor $\Delta_{\lambda} \cong P_{\lambda}$. This is absurd.

We have proved by contradiction that ${\lambda} \in O_1$ and is maximal in $O_1$ with respect to $\leqslant'$. The conclusion follows from induction.
\end{proof}

\subsection{Example illustrating the above algorithm.} The following example illustrates our algorithm.

Let $\mathcal{A}$ be the following algebra with relation $\alpha^2 = \delta^2 = \beta \alpha = \delta \gamma = \rho^2 = \rho \varphi = 0$.
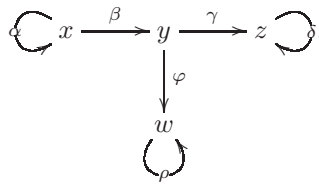
\begin{equation*}
\xymatrix{ x \ar[r] ^{\beta} \ar@(lu, ld)[]|{\alpha} & y \ar[r] ^{\gamma} \ar[d] ^{\varphi} & z \ar@(ru, rd)[]|{\delta} \\
 & w \ar@(dl, dr)[]|{\rho}}.
\end{equation*}
The projective $\mathcal{A}$-modules are:
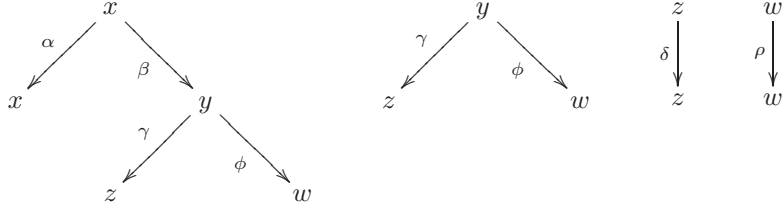
\begin{equation*}
\xymatrix{ & x \ar[dl] _{\alpha} \ar[dr] _{\beta} & & \\ x & & y \ar[dl] _{\gamma} \ar[dr] _{\phi} & \\ & z & & w}
\qquad \xymatrix{ & y \ar[dl] _{\gamma} \ar[dr] _{\phi} & \\ z & & w}
\qquad \xymatrix{ z \ar[d] _{\delta} \\ z}
\qquad \xymatrix{ w \ar[d] _{\rho} \\ w}
\end{equation*}

Then by the above algorithm, $O_1 = \{ x, y \}$ and $x \leqslant' y$ in $O_1$, we should take $y$ as the maximal element. But then $O_2 = \{ x, z, w \}$ and all these elements are maximal in $O_2$ with respect to $\leqslant '$. Thus we get three choices for $O_3$. Similarly, the two elements in each $O_3$ are maximal with respect to $\leqslant'$. In conclusion, 6 linear orders are contained in $\mathcal{L}$: $y \succ x \succ z \succ w$, $y \succ x \succ w \succ z$, $y \succ z \succ x \succ w$, $y \succ z \succ w \succ x$, $y \succ w \succ z \succ x$, and $y \succ w \succ x \succ z$. For all these six linear orders $\mathcal{A}$ is standardly stratified and has the same standard modules. Moreover, the category $\mathcal{F} (_{\preccurlyeq} \Delta)$ is closed under cokernels of monomorphisms.
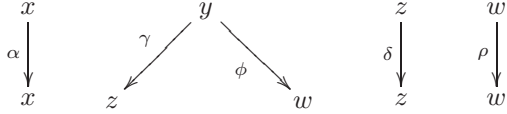
\begin{equation*}
\xymatrix{ x \ar[d] _{\alpha} \\ x}
\qquad \xymatrix{ & y \ar[dl] _{\gamma} \ar[dr] _{\phi} & \\ z & & w}
\qquad \xymatrix{ z \ar[d] _{\delta} \\ z}
\qquad \xymatrix{ w \ar[d] _{\rho} \\ w}
\end{equation*}

\subsection{Different orders in $\mathcal{L}$ determine different standard modules.}

In general, for different linear orders in $\mathcal{L}$ the corresponding standard modules are different.

Let $\mathcal{A}$ be the following category with relation: $\delta^2 = \delta \alpha =0$, $\beta \delta = \beta'$.
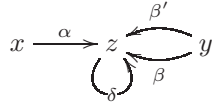
\begin{equation*}
\xymatrix{ x \ar[r] ^{\alpha} & z \ar@(dl,dr)[]|{\delta} & y \ar@/^/[l] ^{\beta} \ar@/_/[l] _{\beta'} }
\end{equation*}

The reader can check that the following two linear orders are contained in $\mathcal{L}$: $x \succ z \succ y$ and $y \succ x \succ z$. The corresponding standard modules are:
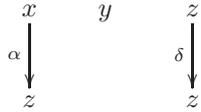
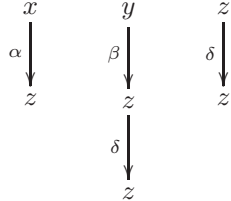
\begin{equation*}
\xymatrix{ x \ar[d] _{\alpha} \\ z} \qquad y \qquad \xymatrix{ z \ar[d] _{\delta} \\ z}
\end{equation*}
and
\begin{equation*}
\xymatrix{ x \ar[d] _{\alpha} \\ z} \qquad \xymatrix{y \ar[d] _{\beta} \\ z \ar[d] _{\delta} \\ z} \qquad \xymatrix{ z \ar[d] _{\delta} \\ z}
\end{equation*}
It is easy to see that $\mathcal{F} (\Delta)$ corresponding to the first order is closed under cokernels of monomorphisms.

\subsection{Not all linear orders can be obtained by the algorithm.}

We reminder the reader that although $A$ is standardly stratified for all linear orders in $\mathcal{L}$, it does not imply that all linear orders for which $A$ is standardly stratified are contained in $\mathcal{L}$. Moreover, it is also wrong that for every linear order $\preccurlyeq \notin \mathcal{L}$, there exists some $\tilde{ \preccurlyeq } \in \mathcal{L}$ such that $\mathcal{F} (_{\preccurlyeq} \Delta) \subseteq \mathcal{F} (_{\tilde {\preccurlyeq}} {\Delta})$. Consider the following example.

Let $A$ be the path algebra of the following quiver with relations: $\delta^2 = \delta \gamma = 0$, $\delta \beta = \beta'$.
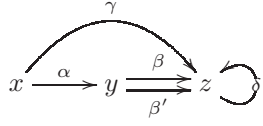
\begin{equation*}
\xymatrix{ x \ar[r]^{\alpha} \ar@/^2pc/[rr] ^{\gamma} & y \ar@<0.5ex>[r] ^{\beta} \ar@<-0.5ex>[r] _{\beta'} & z \ar@(rd,ru)[]|{\delta}}
\end{equation*}
The structures of indecomposable projective $A$-modules are:
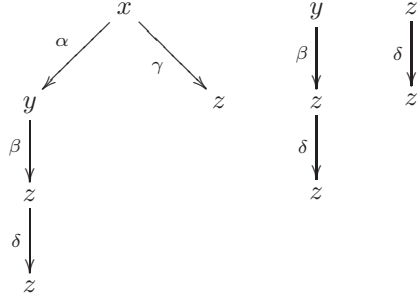
\begin{equation*}
\xymatrix{ & x \ar[dl] _{\alpha} \ar[dr] _{\gamma} & \\ y \ar[d] _{\beta} & & z \\ z \ar[d] _{\delta} & & \\ z & &} \qquad
\xymatrix{y \ar[d] _{\beta} \\ z \ar[d] _{\delta} \\ z} \qquad
\xymatrix{ z \ar[d] _{\delta} \\ z}
\end{equation*}
Applying the algorithm, we get a unique linear order $y \succ x \succ z$ contained in $\mathcal{L}$, and the corresponding standard modules are:
\begin{equation*}
\Delta_x = \xymatrix{ x \ar[d] _{\gamma} \\ z} \qquad \Delta_y \cong P_y \qquad \Delta_z \cong P_z.
\end{equation*}
But there is another linear order $x \succ z \succ y$ for which $A$ is also standardly stratified with standard modules
\begin{equation*}
\Delta_x' \cong P_x \quad \Delta_y' = y \quad \Delta_z' \cong P_z.
\end{equation*}
Both $\mathcal{F} (_{\preccurlyeq} \Delta)$ and $\mathcal{F} (_{\preccurlyeq'} \Delta)$ are not closed under cokernels of monomorphisms. Moreover, each of them is not contained in the other one.

\subsection{$\mathcal{F} (_{\preccurlyeq} \Delta)$ is maximal but not closed under cokernels of monomorphisms.}

In general it is possible that there are more than one linear orders with respect to which $A$ is standardly stratified. However, among these categories $\mathcal{F} (_{\preccurlyeq} \Delta)$, there is at most one which is closed under cokernels of monomorphisms. If this category exists, it is the unique maximal category and contains all other ones as subcategories. But the converse of this statement is not true by the next example. That is, if there is a linear order $\preccurlyeq$ with respect to which $A$ is standardly stratified and $\mathcal{F} (_{\preccurlyeq} \Delta)$ is the unique maximal category, $\mathcal{F} (_{\preccurlyeq} \Delta)$ might not be closed under cokernels of monomorphisms.

Let $A$ be the path algebra of the following quiver with relations: $\gamma^2 = \rho^2 = \delta \gamma = \rho \delta =0$, $\gamma \alpha = \alpha'$, $\gamma \beta = \beta'$, and $\delta \alpha = \delta \beta$.
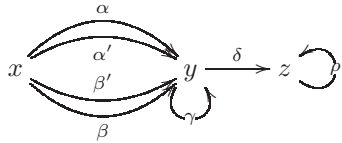
\begin{equation*}
\xymatrix {x \ar@/^1.5pc/[rr] ^{\alpha} \ar@/^1pc/[rr] _{\alpha'} \ar@/_1.5pc/[rr] _{\beta} \ar@/_1pc/[rr] ^{\beta'} & & y \ar@(dl, dr)[]|{\gamma} \ar[r]^{\delta} & z \ar@(rd,ru)[]|{\rho} }
\end{equation*}

The structures of indecomposable projective $A$-modules are described as follows:

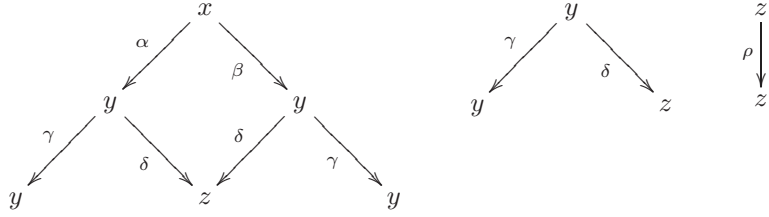
\begin{equation*}
\xymatrix{ & & x \ar[dl] _{\alpha} \ar[dr] _{\beta} & & \\ & y \ar[dl] _{\gamma} \ar[dr] _{\delta} & & y \ar[dl] _{\delta} \ar[dr] _{\gamma}& \\ y & & z & & y } \qquad
\xymatrix{ & y \ar[dl] _{\gamma} \ar[dr] _{\delta} & \\ y & & z} \qquad
\xymatrix{ z \ar[d] _{\rho} \\ z}
\end{equation*}

It is not hard to check that if $A$ is standardly stratified with respect to some linear order $\preccurlyeq$, then all standard modules coincide with indecomposable projective modules. Therefore, $\mathcal{F} (_ {\preccurlyeq} \Delta)$ is the category of all projective modules and is the unique maximal category. However, we get an exact sequence as follows, showing that $\mathcal{F} (\Delta)$ is not closed under cokernels of monomorphisms:
\begin{equation*}
\xymatrix{ 0 \ar[r] & P_y \ar[r] & P_x \ar[r] & M \ar[r] & 0}
\end{equation*}
where $M \notin \mathcal{F} (\Delta)$ has the following structure:
\begin{equation*}
\xymatrix{ x \ar[d] _{\alpha} \\ y\ar[d] _{\gamma} \\ y}
\end{equation*}

\subsection{$\mathcal{F} (_{\preccurlyeq} \Delta)$ is closed under cokernels of monomorphisms but $\mathcal{A}$ is not directed.}

The last statement of Theorem 0.3 tells us that if $A$ is quasi-hereditary and $\mathcal{F} (_{\preccurlyeq} \Delta)$ is closed under cokernels of monomorphisms, then $\mathcal{A}$ is a directed category. This is incorrect if $A$ is supposed to be standardly stratified, as shown by the next example.

Let $A$ be the path algebra of the following quiver with relations $\delta^2 = \delta \alpha = \beta \delta = \beta \alpha = \gamma \beta =0$. Let $x \succ z \succ y$.
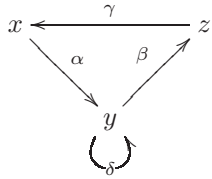
\begin{equation*}
\xymatrix{ x \ar[dr] ^{\alpha} & & z \ar[ll] _{\gamma}\\
& y \ar[ur] ^{\beta} \ar@(dl,dr)[]|{\delta} }
\end{equation*}
Indecomposable projective modules of $A$ are described below:
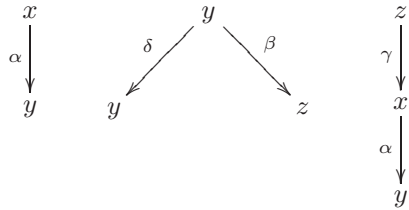
\begin{equation*}
\xymatrix{ x \ar[d] _{\alpha} \\ y} \qquad
\xymatrix{ & y \ar[dl] _{\delta} \ar[dr] ^{\beta} & \\ y & & z}
\qquad \xymatrix{ z \ar[d] _{\gamma} \\ x \ar[d] _{\alpha} \\ y}
\end{equation*}
The standard modules are:
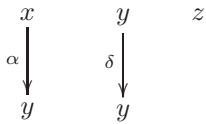
\begin{equation*}
\xymatrix{ x \ar[d] _{\alpha} \\ y} \qquad
\xymatrix{ y \ar[d] _{\delta} \\ y}
\qquad z
\end{equation*}
It is clear that $A$ is standardly stratified. Moreover, by (1) of Theorem 0.3, $\mathcal{F} (\Delta)$ is closed under cokernels of monomorphisms. But $\mathcal{A}$ is not a directed category.

\end{document}